\documentclass[12pt]{article}
\usepackage{amsmath,amsthm,amssymb,amsfonts,mathrsfs}
\usepackage[hypertex, bookmarks, colorlinks=true, plainpages = false,
citecolor = green, urlcolor = blue, filecolor = blue]{hyperref}
\newtheorem{theorem}{Theorem}[section]
\newtheorem{lemma}[theorem]{Lemma}
\newtheorem{proposition}[theorem]{Proposition}
\newtheorem{corollary}[theorem]{Corollary}
\newtheorem{definition}[theorem]{Definition}
\newtheorem{example}[theorem]{Example}

\newtheorem{remark}[theorem]{Remark}

=22cm\headheight=0cm\topskip=0cm\topmargin=0cm
\def\a{\bar {a}}\def\b{\bar {b}}
\def\x{\bar {x}}\def\y{\bar {y}}
\def\z{\bar {z}}

\def\Nn{\mathbb N}

\def\Qn{\mathbb Q}
\def\Rn{\mathbb R}

\def\Sc{\mathcal S}
\def\0{\sf 0}

\makeatletter
\def\dotminussym#1#2{%
  \setbox0=\hbox{$\m@th#1-$}%
  \kern.5\wd0%
  \hbox to 0pt{\hss\hbox{$\m@th#1-$}\hss}%
  \raise.8\ht0\hbox to 0pt{\hss$\m@th#1.$\hss}%
  \kern.5\wd0}

\begin{document}
\begin{center}
{\Large\sc Affine logic with\\ the integration operator}
\bigskip

{{\bf Seyed-Mohammad Bagheri}}
\vspace{5mm}

{\footnotesize {\footnotesize Department of Pure Mathematics,
Faculty of Mathematical Sciences,\\
Tarbiat Modares University, Tehran, Iran, P.O. Box 14115-134}\\
e-mail: bagheri@modares.ac.ir}
\end{center}

\begin{abstract}
Affine continuous logic is extended to affine integration logic. Affine compactness theorem is proved by both the ultramean
construction and Henkin's method. Also, a proof system and a completeness theorem are given.
An appropriate variant of the Keisler-Shelah isomorphism theorem holds in this setting.
This helps us to characterize non-forking extensions in affine stable theories
by means of the notion of elementary embedding in the expanded logic.
\end{abstract}



\section{Introduction}
The interaction between logic and probability is a fascinating subject studied by several authors using different technics.
A related approach in this respect is integration logic originally introduced by Hoover and Keisler \cite{Hoover,Keisler}.
The main object has been to expand first order logic with the integration operator (see also \cite{Bagheri-Pourmahd1}).
In this paper we expand continuous logic with the integration operator.
A tentative study has been done in \cite{Bagheri-Pourmahd2}, where sigma-additive integration theory
is combined with continuous logic. A technical problem in this respect is that sigma-additivity
does not match well with the finitary logical arguments such as Henkin's construction.
Using the ultraproduct method, on the other hand, leads to additional set theoretic requirements.
In the present text, we replace sigma-additive measures with finitely additive measures and concentrate more on the
affine part of continuous logic \cite{ACMT}. The full continuous case has a parallel argument.
ZFC is then sufficient and Heknin's construction works using the mean value property.
Among the significant properties of sigma-additive integration theory are the Fubini property and Fatou's lemma.
The first one is easily stated by affine axioms and the second one is again an infinitary property.
In contrast, an advantage of this relaxation is that types in affine theories are brought into the discussion
by the Riesz representation theorem. More precisely, if $p(x)$ is a type over a continuous structure $M$,
the pair $(M,p)$ can be regarded as a probability metric structure.
The primary objective of continuous integration logic is of course to formulate a logical framework
for stating and studying properties of metric measure structures coming from ergodic theory
or other parts of mathematics.

To sum up, continuous structures are assumed to be equipped with a finitely additive probability measure
on their Borel algebras. It must be however noted that, smaller subalgebras of the Borel algebras
are often sufficient for this purpose. This is because we only need to integrate continuous formulas.
Assuming such an additional information is harmless and does not affect the intended semantics.
Since, every finitely additive probability measure on such a subalgebra is always extended to the
whole Borel algebra without changing integrals.
In this manner, we can have a unified discussion for such probability structures.
An unpleasant feature of this assumption is of course that structures equipped with
distinct measures may be model theoretically isomorphic, hence regarded as identical.

Let $\mathcal A$ be a Boolean algebra of subsets of $X$. A \emph{probability charge} on $\mathcal A$
is any finitely additive function $\mu:\mathcal A\rightarrow[0,1]$ such that $\mu(X)=1$.
We often use the word \emph{charge} to mean a probability charge. If $\mathcal A$ is the power set of $X$,
we call it an \emph{ultracharge}.
Let $(X,\mathcal A,\mu)$ be a charge space and $\mathcal H(\Rn)$ be the Boolean algebra of subsets of $\Rn$
generated by the half-intervals $[r,s)$.
A function $f:X\rightarrow\Rn$ is called $(\mathcal A,\mathcal H(\Rn))$-measurable
if $f^{-1}(A)\in\mathcal A$ for every $A\in\mathcal H(\Rn)$.
Bounded $(\mathcal A,\mathcal H(\Rn))$-measurable functions are always integrable (\cite{Aliprantis-Inf} Th. 11.8).
This fact remains true if we replace $\mathcal{H}(\Rn)$ with the Boolean algebra $\mathcal{H}_{\Qn}(\Rn)$
generated by rational half-intervals, i.e. intervals $[r,s)$ where $r,s\in \Qn$.

Let $M$ be a metric space. $\mathcal A(M)$ denotes the \emph{Borel algebra}\index{Borel algebra}
of $M$, i.e. the Boolean algebra generated by the open subsets of $M$ and $\mathcal B(M)$ denotes the \emph{Borel sigma-algebra} of $M$.
A charge defined on $\mathcal A(M)$ is called a \emph{Borel charge}\index{Borel charge}.
A Borel charge $\mu$ is called \emph{regular} if it is both inner and outer regular,
i.e. for every $A\in\mathcal{A}(M)$
$$\mu(A)=\sup\{\mu(F)|\ A\supseteq F\in\mathcal A(M),\ F\ \mbox{is\ closed}\}$$
$$\mu(A)=\inf\{\mu(G)|\ A\subseteq G\in\mathcal A(M),\ G\ \mbox{is\ open}\}.$$
In fact, since $\mu(M)$ is finite, any one of these properties implies the other one.
If $\mu$ is a Borel charge on $M$, every bounded continuous real valued function on $M$ is $\mu$-integrable.

In this paper, we use integration for two completely different purposes.
First, if $M$ is a metric structure and $\mathcal A$ is its Borel algebra, then we integrate
continuous formulas in order to create new continuous formulas.
Second, if $X=I$ is an index set and $\mathcal A$ is its power set, we integrate functions
$f:I\rightarrow\Rn$ in order to construct \emph{ultrameans}. We will use the notation $\int_I$ in the second case.

\section{Syntax and semantics}
Pseudometric spaces are denoted by $X, M, N$ etc. If there is no risk of confusion,
their metrics as well as the corresponding quotient metrics are all denoted by the symbol $\rho$.
If $M$ is a pseudometric space, we put the pseudometric
$\sum_i \rho(a_i,b_i)$ on $M^n$. A function $f:M\rightarrow N$ is $\lambda$-Lipschitz if
$$\hspace{10mm} \rho(f(x),f(y))\leqslant\lambda \rho(x,y)\ \ \ \ \ \ \ \forall x,y\in M.$$

We use $\Rn$ as the \emph{space of truth values} and $1$, addition $+$ and scalar multiplications $r\cdot$
(for any $r\in\Rn$) as connectives in affine logic.
Also, $\inf$, $\sup$, $\int$ are quantifiers and $x,y,z,...$ is a list of individual variables.
In full continuous logic (CL), one further uses the connectives $\wedge$ and $\vee$.
This family of symbols generates continuous integration logic.

A {\em Lipschitz language} is a first order language $L=\{F,...; R,...; c,...\}$
such that to each function symbol $F$ is assigned a Lipschitz constant $\lambda_F\geqslant0$
and to each relation symbol $R$ is assigned a Lipschitz constant $\lambda_R\geqslant0$ .
We have also an additional symbol $\rho$ for the metric (with $\lambda_{\rho}=1$)
which is traditionally put in the list of logical symbols as in first order logic.
A charged language consists of a Lipschitz language $L$ augmented with a probability charge symbol $\mu$
which is put again in the list of logical symbols. We may generally allow several such symbols which
are to be interpreted by Borel charges on the intended structure $M$.
Even, we may allow n-dimensional symbols which can be interpreted by Borel charges on $M^n$.
The arguments which will follow are then similar.
In this paper, we only use one unary charge symbol $\mu$. Sometimes, to emphasize that the language is charged we denote it by $L_\mu$.

\begin{definition} \label{prestructure}
{\em A {\em prestructure} in $L$ is a pseudo-metric space $(M,\rho)$ equipped with:

- for each $c\in L$, an element $c^{M}\in M$

- for each $n$-ary function symbol $F$ a $\lambda_F$-Lipschitz function
$F^{M}:M^n\rightarrow M$

- for $n$-ary relation symbol $R$ (including $\rho$) a $\lambda_R$-Lipschitz function
$R^{M}:M^n\rightarrow[0,1]$

- a regular Borel probability charge $\mu$ on $M$.}
\end{definition}
\vspace{2mm}

It is well-known that if $M$ is a compact Hausdorff space, every regular Borel charge on
$M$ is uniquely extended to a regular Borel measure on $M$ (see \cite{Aliprantis-Inf} \S 14).
So, for compact prestructures, $\mu$ may be assumed to be a regular Borel measure.
\vspace{1mm}

\emph{$L$-terms} and their Lipschitz constants are inductively defined as follows:

\begin{quote}
- Every $c\in L$ and variable $x$ is a term with Lipschitz constants resp $0$ and $1$.

- If $F$ is a $n$-ary function symbol and $t_1,...,t_n$ are terms, then $t=F(t_1,...,t_n)$ is a
term with Lipschitz constant $\lambda_t=\lambda_F\cdot\sum_i\lambda_{t_i}$.
\end{quote}
One checks that for any $M$ and term $t(\x)$,\ \ $t^M:M^n\rightarrow M$ is $\lambda_t$-Lipschitz.
$L$-formulas in \emph{affine integration logic} (AL$_{\int}$-formulas) and their Lipschitz constants and bounds
are inductively defined as follows:

\begin{quote}
- $1$ is an atomic formula with Lipschitz constant $0$ and bound $1$.

- If $R$ is a $n$-ary relation symbol (including $\rho$ with $n=2$) and $t_1,...,t_n$ are terms,
then $R(t_1,...,t_n)$ is an atomic formula with Lipschitz constant
$\lambda_R\cdot\sum\lambda_{t_i}$ and bound ${\sf{b}}_R=1$.

- If $\phi,\psi$ are formulas and $r\in\Rn$, then $\phi+\psi$ and $r\phi$ are formulas
with Lipschitz constants $\lambda_\phi+\lambda_\psi$,\ $|r|\lambda_\phi$ and bounds
${\sf{b}}_{\phi}+{\sf{b}}_\psi$,\ $|r|{\sf{b}}_\phi$ respectively.

- If $\phi$ is a formula and $x$ is a variable, then $\inf_x\phi$, $\sup_x\phi$ and $\int\phi\ dx$
are formulas with the same Lipschitz constant and bound as $\phi$.
\end{quote}
\vspace{2mm}

Restricting AL$_{\int}$ to integration-free formulas, one obtains affine logic AL \cite{ACMT}.
In this text, by formula we mean an AL$_{\int}$-formula.
The notion of free variable is defined in the usual way
regarding the symbols $\sup$, $\inf$ and $\int$ as unary quantifiers.
As usual, one writes $\phi(\x)$ to indicate that $\x$ contains the free variables of $\phi$.
A sentence is a formula without free variables.

\begin{definition}
{\em Let $M$ be a $L$-prestructure. The value of a formula $\phi(\x)$ in $\a\in M$ is inductively defined as follows:

- $(R(t_1,...,t_n))^M(\a)=R^M(t^M_1(\a),...,t_n^M(\a))$

- $(r\phi+s\psi)^M(\a)=r\phi^M(\a)+s\psi^M(\a)$

- $(\inf_y\phi)^M(\a)=\inf_{b\in M}\phi^M(\a,b)$\hspace{11mm} (where $\phi=\phi(\x,y)$)

- $(\sup_y\phi)^M(\a)=\sup_{b\in M}\phi^M(\a,b)$ \hspace{25mm} ...

- $(\int\phi\ dy)^M(\a)=\int\phi^M(\a,y)\ d\mu$ \hspace{28mm} ...}
\end{definition}

One proves easily that for every $L$-prestructure $M$,\ \ $\phi^M(\x)$ is a
$\lambda_{\phi}$-Lipschitz function on $M^n$ and $|\phi^M(\a)|\leqslant{\sf{b}}_{\phi}$ for all $\a$.
The following two theorems are essential in our discussion.

\begin{theorem} \emph{(Kantorovich, \cite{Aliprantis-Inf} Th. 8.32)} \label{ext1}
Let $G$ be a majorizing vector subspace of a Riesz space $E$.
Let $\Lambda:G\rightarrow\Rn$ be a positive linear function.
Then $\Lambda$ has an extension to a positive linear function on $E$.
\end{theorem}

\begin{theorem} (Riesz representation theorem \cite{Aliprantis-Inf}, Th.14.9)\label{Riesz representation} 
Let $X$ be a normal Hausdorff topological space and $\Lambda:{\mathbf C}_b(X)\rightarrow\Rn$ be a positive linear functional.
Then there is a unique finite regular Borel charge $\mu$ such that
$\mu(X)=\|\Lambda\|=\Lambda(1)$ and $$\ \ \ \ \Lambda(f)=\int fd\mu\hspace{16mm} \forall f\in\mathbf{C}_b(X).$$
\end{theorem}

We will then use the following corollary frequently in order to find Borel charges on structures.

\begin{corollary}\label{KR}
Let $X$ be a metric space and $G\subseteq{\mathbf C}_b(X)$ be a majorizing vector subspace.
Then, for every positive linear $\Lambda:G\rightarrow\Rn$, there is a regular Borel charge $\mu$ on $X$
such that $\mu(X)=\Lambda(1)$ and $$\Lambda(f)=\int fd\mu\hspace{18mm} \forall f\in G.$$
\end{corollary}

An $L$-structure is a prestructure $(M,\rho,\mu)$ for which $\rho$ is a complete metric.
Every prestructure is transformed to a $L$-structure in the natural way.
We briefly review the argument. 
Let $M$ be a prestructure in $L$ and $N$ be the corresponding quotient metric space.
Let $\pi:M\rightarrow N$ be the quotient map. It is easily verified that
$$\mathcal A(N)=\{X\subseteq N|\ \pi^{-1}(X)\in\mathcal A(M)\}.$$
Interpretations of function and relation symbols on $M$ induce well-defined functions and relations on $N$.
Also, $\pi$ induces a Borel charge $\nu$ on
$N$ by
$$\nu(X)=\mu(\pi^{-1}(X)), \ \ \ \ \ \ \ X\in\mathcal A(N).$$
Then, $(N,\nu)$ is a $L$-prestructure and
$$\ \ \int f\ d\nu=\int f\circ\pi\ d\mu\hspace{16mm} \forall f\in\mathbf{C}_b(N).$$
We must now extend this prestructure to the completion $\bar{N}$ of $N$.
Clearly, the interpretation of every function (resp. relation) symbol induces
a unique well-defined function (resp. relation) on $\bar{N}$.
Also, $\nu$ induces a positive linear functional $\bar\Lambda$ on $\mathbf{C}_b(\bar{N})$ by $$\bar\Lambda(\bar f)=\int fd\nu$$
where $\bar{f}$ is the natural extension of $f$ to $\bar{N}$.
So, by Corollary \ref{KR}, there is a regular Borel charge $\bar\nu$ on $\bar{N}$
such that
$$\int f\ d\nu=\bar\Lambda(\bar f)=\int \bar f\ d\bar\nu\hspace{14mm} \forall f\in\mathbf{C}_b(N).\hspace{10mm}(*)$$
We can now state the main result on the completion of Lipschitz prestructures.

\begin{proposition}\label{exist1}
For each $L$-prestructure $M$ there is a $L$-structure $K$ and a dense map
$\pi:M\rightarrow K$ such that for every formula $\phi(\x)$ and $a_1,...,a_k\in M$
$$\phi^{M}(a_1,...,a_k)=\phi^{K}(\pi a_1,...,\pi a_k).$$
\end{proposition}
\begin{proof} With the above notations, set $K=\bar{N}$. The claim is proved by induction on the complexity of $\phi$.
We verify the integral case assuming it is proved for $\phi(x)$. Let $f$ be the map induced by $\phi^M(x)$ on $K$
(via $M\stackrel{\pi}{\rightarrow}N\stackrel{id}{\rightarrow}K$).
Then, by the induction hypothesis, for all $x\in M$, we have that $f(\pi x)=\phi^M(x)=\phi^{K}(\pi x)$
and hence $\bar f=\phi^{K}$. So, by $(*)$ above
$$\Big(\int\phi(x) dx\Big)^{K}=\int\phi^{K}(x)d\bar\nu=\int\bar{f}d\bar{\nu}$$
$$=\int f\ d\nu=\int f\circ\pi\ d\mu=\Big(\int\phi(x) dx\Big)^M.$$
\end{proof}

\begin{example}
\em{1. Every compact metric space $(M,\rho)$ equipped with a Borel probability measure $\mu$
is a charged metric structure in the empty language.
If $(M,\cdot)$ is a group and $\rho$ is left invariant, then, $(M,\rho,\cdot,\mu)$ is a
charged metric structure. If $f:M\rightarrow M$ is Lipschitz, then $(M,\rho,f,\mu)$ is a charged metric structure.
\vspace{2mm}

\noindent 2. Let $M$ be a first order model. Then, the Borel algebra of $M$ coincides the power set
of $M$ and a Borel charge $\mu$ on $M$ is nothing else than an ultracharge (which is necessarily regular).
So, $(M,\mu)$ is a charged metric structure.
An interesting example in this respect is $(\Nn,+,\cdot)$ equipped with any ultracharge $\mu$.
For example, if $\mathcal F$ is an ultrafilter on $\Nn$, then
$$\ \ \mu(A)=\lim_{n,\mathcal F}\frac{|A\cap n|}{n}\hspace{15mm} (A\subseteq\Nn)$$
defines an ultracharge which extends the natural density.
First order models equipped with an ultrafilter are studied in \cite{Veloso} under the name of \emph{ultrafilter logic}.
So, continuous integration logic extends ultrafilter logic.
Every first order structure $(M,a)$ can be regarded as a charged metric structure and
$\int\phi^M(x)dx=\phi^M(a)$ for every $\phi(x)$.
A more natural way of viewing a metric structure $M$ as a charged one, is to allow $0\leqslant\mu(M)\leqslant1$.
Then, $\mu(M)=0$ corresponds to the uncharged case.}
\end{example}

Two $L$-structures $M$, $N$ are equivalent, denoted by $M\equiv N$, if $\sigma^M=\sigma^N$ for every sentence $\sigma$.
A map $f:M\rightarrow N$ is an \emph{embedding} if $\phi^M(\a)=\phi^N(f(\a))$ for every quantifier-free formula $\phi(\x)$ and $\a\in M$.
It is an \emph{elementary embedding} if $\phi^M(\a)=\phi^N(f(\a))$ for every formula $\phi(\x)$ and $\a\in M$.
An \emph{isomorphism} is a surjective elementary embedding. So, an isomorphism is not necessarily measure preserving.
This is because we have usually a collection of measurable subsets irrelevant to the semantics.
Meanwhile, every surjective measure preserving embedding is an isomorphism. 
If $f=id$ is an embedding (resp. elementary embedding), we say $M$ is a \emph{substructure}
(resp. an \emph{elementary substructure}) of $N$ and denoted this by $M\subseteq N$ (resp $M\preccurlyeq N$).
Again, for being a substructure, we do not require any organic relation between the charges of $M$ and $N$.
It is however natural to induce every pure metric substructure $M$ of $(N,\mathcal{B},\mu)$
with the \emph{subspace charge} defined by
$$\mathcal A=\{Y\cap M:\ Y\in\mathcal B\}\ \ \ \ $$
$$\ \ \ \ \ \ \ \ \ \ \ \nu_M(X)=\inf\{\nu(Y):\ X\subseteq Y\in\mathcal{B}\}\ \ \ \ \ \ \ \ \ \ \mbox{for\ any}\ X\in\mathcal A.$$
Then, $\nu_M\leqslant1$. This is an other reason why me must generally allow $\mu(N)\leqslant1$.
In this case, usually a constant coefficient appears in computations.
The following lemma shows that the regularity assumption for $\mu$ is not a serious restriction.

\begin{lemma}
Let $M$ be a metric space equipped with a charge $\mu$ on $\mathcal A\subseteq\mathcal{A}(M)$.
Let $G\subseteq\mathbf{C}_b$ be a vector space of $\mu$-integrable functions.
Then, there is a regular Borel charge $\nu$ on $M$ such that $\int fd\mu=\int fd\nu$ for all $f\in G$.
\end{lemma}
\begin{proof}
We may further assume $G$ contains constant maps. Let $\Lambda_0(f)=\int fd\mu$ for every $f\in G$. By Corollary \ref{KR},
$\Lambda$ is represented by a regular Borel charge $\nu$ on $M$. We have therefore that $\int fd\mu=\Lambda(f)=\int fd\nu$ for every $f\in G$.
\end{proof}

In particular, it is quite possible that $(M,\mu)\simeq(M,\nu)$ via the identity map
for two distinct Borel charges $\mu$ and $\nu$.

\begin{proposition}\label{chain2}
Let $(M_i)_{i<\kappa}$ be an elementary chain of $L$-structures.
Then there is a $L$-structure on the completion $\bar M$ of $M=\bigcup_i{M_i}$ such that $M_j\preccurlyeq\bar M$ for each $j$.
\end{proposition}
\begin{proof} The metric structure of $M$ is defined in the natural way,
e.g. $R^M=\cup_i R^{M_i}$. Let $\phi(\x,y)$ be a formula and $\a\in M_i$.
Then, $\phi^{M_j}(\a,y)$ extends $\phi^{M_i}(\a,y)$ whenever $i\leqslant j$.
Let $G$ be the family of all functions $\cup_i\phi^{M_i}(\a,y)$ where $\a\in\cup_{i} M_i$.
This is a majorating subspace of $\mathbf{C}_b(M)$.
If $\a\in M_j$, set $$\Lambda(\cup_i\phi^{M_i}(\a,y))=\int \phi^{M_j}(\a,y)\ d\mu_j.$$
By the assumptions, $\Lambda_0$ is a well-defined positive linear functional on $G$. Apply Corollary \ref{KR}
to find a regular Borel charge $\mu$ on $\mathcal A(M)$ representing $\Lambda$.
So, for every formula $\phi(\x,y)$ and $\a\in M_j$, we have that $$\int \phi^{M_j}(\a,y)\ d\mu_j=\int\cup_i\phi^{M_i}(\a,y)\ d\mu.$$
We can now prove that $M_i\preccurlyeq M$ for each $i$. For this purpose, we show by induction
on the complexity of $\phi$ that for each $i$, whenever $\a\in M_i$, one has that $\phi^{M_i}(\a)=\phi^M(\a)$.
We check the integral case. Suppose the claim holds for $\phi(\x,y)$. Then, if $\a\in M_j$,
$$\int\phi^M(\a,x)d\mu=\int(\cup_i\phi^{M_i}(\a,x))d\mu=\int\phi^{M_j}(\a,x)d\mu_j.$$
Now, use Proposition \ref{exist1}
\end{proof}

Let $(N,\mathcal B,\nu)$ be a structure and $M$ be a metric substructure of $N$.
Let $\nu_M$ be the subspace measure defined above.
If every $Y\in\mathcal B$ with positive measure has non-empty intersection with $M$,
then $\nu(Y)=\nu_M(Y\cap M)$ for every such $Y$.
Also, in this case, if $f:N\rightarrow\Rn$ is bounded and $(\mathcal B,\mathcal H(\Rn))$-measurable,
then $$\int fd\nu=\int f|_M\ d\nu_M.$$

\begin{proposition}\label{downward2} \emph{(Downward)}
Let $|L|+\aleph_0\leqslant\kappa$. Assume $(N,\nu)$ is a $L$-structure and $X\subseteq N$
has density character $\leqslant\kappa$.
Then, there exists $X\subseteq M\preccurlyeq N$ such that $\mathsf{dc}(M)\leqslant\kappa$.
\end{proposition}
\begin{proof}
We construct chains $$M_0\subseteq M_1\subseteq\ \cdots\ \subseteq N$$
$$\mathcal{C}_0\subseteq\mathcal{C}_1\subseteq\ \cdots\ \subseteq\mathcal{A}(N)$$
of incomplete metric substructures and Boolean subalgebras such that $|M_k|\leqslant\kappa$ and $|\mathcal{C}_k|\leqslant\kappa$.
We may assume without loss that $X\subseteq M_0$.
Also, set $\mathcal{C}_0=\{\emptyset,N\}$. Suppose that $M_k$ and $\mathcal{C}_k$ are constructed.
We construct $M_{k+1}$ and $\mathcal{C}_{k+1}$ as follows.
For each $X\in\mathcal{C}_k$ of positive measure take $a_X\in N$. Also, for each
$n\geqslant1$ and $f(x)=\phi^N(\a,x)$ where $\phi$ is a rational formula and $\a\in M_k$
take $b_{nf}\in N$ such that $$\sup_x f(x)-\frac{1}{n}\leqslant f(b_{nf}).$$
Let $M_{k+1}$ be the metric prestructure of $N$ generated by $M_k$ and the points $a_X$, $b_{f,n}$ defined so.
Also, let $\mathcal{C}_{k+1}\subseteq\mathcal{A}(N)$ be the Boolean subalgebra
generated by sets of the form $f^{-1}([0,\infty))$ where $f$ is as above.
Set $$M=\overline{\bigcup_k M_k},\ \ \ \ \ \ \ \ \ \mathcal{C}=\bigcup_k\mathcal{C}_k.$$
Note that every $A\in\mathcal{C}$ of positive measure has non-empty intersection with $M$.
Let $\mu$ be any extension of $(\nu|_{\mathcal C})_M$ to a Borel charge on $M$.
Then, $(M,\mu)$ is a structure in $L$. We show that for every $L$-formula $\phi$ and $\a\in M$,
one has that $\phi^M(\a)=\phi^N(\a)$. This is done by induction on the complexity of $\phi$.
We verify the integral case. Suppose that the claim holds for $\phi(y)$.
Then, for every $r<s$, the sets $(\phi^M)^{-1}([r,s))$ and $(\phi^N)^{-1}([r,s))$
have the same measure. So, by the definition of integral, we have that $\int\phi^M(y)d\mu=\int\phi^N(y)d\nu$.
\end{proof}

\begin{proposition}\label{tarski} \emph{(Tarski-Vaught's test)}
Assume $M\subseteq N$ and for every formula $\phi(\a,y)$ where $\a\in M$ one has that\\
(i) $\inf\{\phi^N(\bar{a},c)\ | \ c\in N\}=\inf\{\phi^N(\bar{a},b) \ | \ b\in M\}$\\
(ii) For any rational $r_i,s_i$ and $\a\in M$, if $\bigcap_{i=1}^n\phi_i^N(\a,y)^{-1}([r_i,s_i))$
is nonempty, it has nonempty intersection with $M$.\\ Then, $M\preccurlyeq N$.
\end{proposition}
\begin{proof}
We prove by induction on the complexity of $\phi(\x)$ that for each $\a\in M$, $\phi^M(\a)=\phi^N(\a)$.
Let us verify the integral case. Assume the claim holds for $\psi(\x,y)$.
Let $\mathcal B$ be the Boolean algebra generated by set of the form $\phi^N(\a,y)^{-1}([r_i,s_i))$
where $\phi$ is a formula and $\a\in N$. The condition (ii) implies that any $Y\in\mathcal B$
with positive measure has non-empty intersection with $M$. Hence, $\int\psi^M(\a,y)dy=\int\psi^N(\a,y)dy$ for any $\a$.
\end{proof}

\section{Affine compactness}\label{section compactness}
Let $L$ be a language. A condition is an expression of the form $\phi\leqslant\psi$ where $\phi,\psi$ are formulas.
A condition $\phi\leqslant\psi$ where $\phi$, $\psi$ are sentences is satisfiable if there is a structure $M$
such that $\phi^M\leqslant\psi^M$. Any set of such conditions is called a theory. A theory $T$ is satisfiable if
there is a structure satisfying every condition in $T$. The notion $T\vDash\phi\leqslant\psi$ is defined as usual.
We give some easy examples.

\begin{example}
\em{\noindent 1. Let $L=\{+,\cdot,0,1\}$ and $\mu$ be an extension of the natural density to an ultracharge on $\Nn$.
Then, the condition $$\int(1-\inf_{yz}\rho(x,y^2z))dx=\frac{6}{\pi^2}$$
states that the density of the set of square-free numbers is $\frac{6}{\pi^2}$.
Similarly, that the density of prime numbers is zero is expressible.

\noindent2. Let $L=\{\wedge,\vee,0,1\}$ and $M=\{0,1\}$ equipped with the discrete metric and charge
$$\mu(0)=\mu(1)=\frac{1}{2}.$$
The metric theory of $M$ eliminates the quantifiers $\sup$-$\inf$ and every metric formula $\phi(\x,y)$
is equivalent to a linear combination of formulas $\rho(t(\x),y)$ where $t(\x)$ is term.
So, since $$\int\rho(x,y)dy=\frac{1}{2},$$ the integral theory of $M$ eliminates the quantifier $\int$.
We conclude that the AL$_{\int}$-theory of $M$ has the same definable relations as the AL-theory of $M$.
\bigskip

\noindent3. Consider the unit interval $[0,1]$ equipped with the Euclidean metric and Lebesgue
measure. Then $\int \rho(x,y)dy=x^2-x+\frac{1}{2}$. So, $x(1-x)$ is a definable relation here.
\bigskip

\noindent4. The condition $\inf_x\int\rho(x,y)dy=0$ states that $\mu$ is a filter containing
the family of open balls around some $a$.
To see this, for each $n$, take $a_n$ such that $\int\rho(a_n,y)dy<\frac{1}{2^{n+1}}$. Then,
$$\mu(B(a_n,\frac{1}{2^n}))\geqslant1-\frac{1}{2^n}.$$
So, we must necessarily have that $d(a_n,a_{n+1})<\frac{1}{2^{n-1}}$. If $a_n$ converges to $a$,
every open neighborhood of $a$ has charge $1$.}
\end{example}
\bigskip

A theory $T$ is affinely satisfiable if for every conditions $\phi_i\leqslant\psi_i$ in $T$ where $i=1,...,k$
and every $r_1,...,r_k\geqslant0$, the condition $\sum_ir_i\phi_i\leqslant\sum_i\ r_i\psi_i$ is satisfiable.
We prove affine compactness by both ultramean and Henkin's method.
\bigskip

\noindent{\bf Ultramean construction}\\
Let $I$ be a nonempty index set and $\wp$ an ultracharge on $I$.
For each $i\in I$, let $(M_i,\rho_i,\mu_i)$ be a $L$-structure.
For $a,b\in \prod_{i\in I} M_i$ set $$\rho(a,b)=\int_I \rho_i(a_i,b_i)\ d\wp.$$
Recall that $\int_I$ denotes integration with respect to $\wp$ and $\int$ integration with respect $\mu$.
Clearly, $\rho$ is a pseudometric and $\rho(a,b)=0$ is an equivalence relation.
We denote the resulting quotient space by $M=\prod_{\wp} M_i$ and
the induced metric on it again by $\rho$.
Also, the equivalence class of $(a_i)$ is denoted by $[a_i]$.
Define a $L$-prestructure on $M$ as follows.
For each constant symbol $c$, $n$-ary function symbol $F$ and $n$-ary relation symbol $R$ set
$$c^M=[c^{M_i}]$$
$$F^M([a_i^1],...,[a_i^n])=[F^{M_i}(a_i^1,...,a)i^n)]$$
$$R^M([a_i^1],...,[a_i^n])=\int_I R^{M_i}(a_i^1,...,a_i^n)\ d\wp.$$
Then, $M$ is a (incomplete) metric structure in $L$. Now, we define a Borel charge on $M$.

Assume $f_i:M_i\rightarrow\Rn$ is $\lambda$-Lipschitz and $\|f_i\|_{\infty}\leqslant \mathbf b$ for every $i\in I$.
Then the function $$f([x_i])=\int_I f_i(x_i)\ d\wp$$
is a well-defined $\lambda$-Lipschitz function on $M$ and $\|f\|_{\infty}\leqslant\mathbf b$.
Denote this function by $[f_i]$. The family of all functions of the form $[f_i]$
(for some $\lambda$ and $\mathbf b$), forms a majorizing linear subspace $G$ of $\mathbf{C}_b(M)$.
In particular, $$r[f_i]+s[g_i]=[rf_i+sg_i].$$ Define a functional $\Lambda$ by setting for each $[f_i]\in G$
$$\Lambda([f_i])=\int_I\int f_id\mu_i d\wp.$$
We show that $\Lambda$ is well-defined. 
Equivalently, if $[f_i]=0$ then $\int_I\int f_id\mu_id\wp=0$.
For this purpose, it is sufficient to prove that $\Lambda$ is positive, i.e.
$$\Big[\forall(x_i)\ \ \ \ \ 0\leqslant\int_I f_i(x_i)d\wp\Big]\ \ \ \Longrightarrow\ \ \ 0\leqslant\int_I\int f_id\mu_i d\wp.$$
Assume $r_i=\inf_xf_i(x)$. Given $\epsilon>0$, choose $x_i\in M_i$ such that $f_i(x_i)\leqslant r_i+\epsilon$.
Then $$0\leqslant\int_I f_i(x_i)d\wp\leqslant\int_I r_id\wp+\epsilon.$$
This shows that $0\leqslant\int_I r_id\wp$.
Now, since $r_i\leqslant f_i(x_i)$ for any $x_i\in M_i$, we conclude that for each $i$,
$$r_i\leqslant\int f_i d\mu_i$$ and hence $$0\leqslant\int_I\int f_id\mu_i d\wp.$$
The linearity of $\Lambda$ is obvious.

By Corollary \ref{KR}, 
we get a regular Borel charge $\mu$ on $\mathcal A(M)$
such that for every $f=[f_i]$ $$\int f\ d\mu=\Lambda(f)=\int_I\int f_id\mu_i d\wp.$$
To summarize, writing $x=[x_i]$, we have that
$$\int \int_I f_i(x_i)d\wp \ d\mu=\int_I\int f_i(x_i)d\mu_i  d\wp. \ \ \ \ \ \ \ \ \ (\dag)$$

\begin{theorem} \emph{(Ultramean)}
For every AL$_{\int}$-formula $\phi(x_1,...x_n)$ and $[a^1_i],...,[a^n_i]$,
$$\phi^M([a^1_i],...,[a^n_i])=\int_I\phi^{M_i}(a^1_i,...,a^n_i)\ d\wp.$$
\end{theorem}
\begin{proof} We prove the claim by induction on the complexity of formulas.
The atomic as well as the connective cases are obvious.
Assume the claim is proved for $\psi(\x,y)$ and $\phi=\sup_y\psi(\x,y)$. We assume $|\x|=0$.
Given $\epsilon>0$, take $b_i$ such that $\phi^{M_i}-\epsilon\leqslant\psi^{M_i}(b_i)$.
Then $$\int_I\phi^{M_i}d\wp-\epsilon\leqslant\int_I\psi^{M_i}(b_i)d\wp=\psi^M(b)\leqslant\phi^{M}.$$
Conversely, given $\epsilon>0$, take $b$ such that $\phi^M-\epsilon\leqslant\psi^M(b)$.
Then $$\phi^{M}-\epsilon\leqslant\psi^M(b)=\int_I\psi^{M_i}(b_i)d\wp\leqslant\int_I\phi^{M_i}d\wp.$$
We conclude that $\phi^M=\int_I\phi^{M_i}d\wp$.
Similarly, we prove the claim for $\phi=\int\psi(y)dy$ assuming it is proved for $\psi(y)$.
By the induction hypothesis, for each $y=[y_i]\in M$
$$\psi^M(y)=\int_I\psi^{M_i}(y_i)d\wp=[\psi^{M_i}]([y_i]).$$
Therefore, by the equality $(\dag)$,
$$\phi^M=\int\psi^M(y)\ d\mu=\int\int_I\psi^{M_i}(y_i)d\wp\ d\mu$$
$$=\int_I\int \psi^{M_i}(y_i)d\mu_i\ d\wp=\int_I\phi^{M_i}d\wp.$$
\end{proof}

\begin{remark}
\emph{If $\mathcal F$ is an ultrafilter, one proves in a similar way that for every CL$_{\int}$-formula $\phi(\x)$
one has that $$\phi^M([a^1_i],...,[a^n_i])=\lim_{i,\mathcal F}\phi^{M_i}(a^1_i,...,a^n_i).$$}
\end{remark}

If $M_i=N$ for every $i$, then $N^\wp=\prod_\wp M_i$ is the \emph{powermean} of $N$ denoted by $N^\mu$.

\begin{corollary}
The diagonal map $a\mapsto[a]$ is an elementary embedding of $N$ into $N^\wp$.
\end{corollary}

As in the AL case, one has that

\begin{theorem} \label{affine compactness}
(Affine compactness) Every affinely satisfiable theory $\Sigma$ is satisfiable.
\end{theorem}
\begin{proof}
We may assume $\Sigma$ is maximal. For each $\sigma$ let $$q(\sigma)=\inf\{r:\sigma\leqslant r\in\Sigma\}.$$
Then, $p$ is total and sublinear, i.e.
$$q(\sigma+\eta)\leqslant q(\sigma)+p(\eta), \hspace{15mm} q(r\sigma)=rq(\sigma)\ \ \ \ (r\geqslant0).$$
Let $T_0$ be the identity map on $\Rn\subseteq\mathbb{D}(L)$ (the normed vector space of formulas).
Then, $T_0\leqslant q$ on $\Rn$ and by the Hahn-Banach extension theorem (\cite{Aliprantis-Inf}, Th 8.30),
$T_0$ extends to a linear map $T:\mathbb{D}(L)\rightarrow\Rn$ such that $T(\sigma)\leqslant q(\sigma)$ for every $\sigma$.
Note that $T$ is positive.

Let $\{M_i\}_{i\in I}$ be a set containing a model from each equivalence class of the relation $M\equiv N$.
We put the discrete topology on $I$. So, we may write $\Rn\subseteq\mathbb{D}(L)\subseteq\mathbf{C}_b(I)$
if we identify $\sigma\in\mathbb{D}(L)$ with the map $i\mapsto\sigma^{M_i}$.
Clearly, $\mathbb{D}(L)$ majorizes $\mathbf{C}_b(I)$. So, by Corollary{KR},
there is a (maximal) charge $\mu$ on $I$ such that for every $\sigma$ $$T(\sigma)=\int_I\sigma^{M_i}d\mu.$$
Let $N=\prod_{\mu}M_i$. Then, for every $\sigma$ $$\sigma^N=\int_I\sigma^{M_i}d\mu=T(\sigma)\leqslant p(\sigma).$$
Finally, if $\sigma\leqslant\eta\in\Sigma$, then $\sigma^N-\eta^N\leqslant q(0)=0$ and hence $N$ is a model of $\Sigma$.
\end{proof}

A class $\mathcal K$ of $L$-structures is elementary if there is a theory $T$ such that $$\mathcal K=\{M:\ M\vDash T\}.$$
The following propositions are proved as in affine logic.

\begin{proposition} \emph{(Axiomatizability)}
A class $\mathcal K$ is elementary if and only if it is closed under
the ultramean construction and affine equivalence.
\end{proposition}

\begin{proposition}
If $|L|+\aleph_0\leqslant\kappa$, every theory which has a model of cardinality
at least $2$ has a model of density character $\kappa$.
\end{proposition}

If $M$ is first order, then for every $a$, $$\mu(a)=\int(1-\rho(x,a))dx.$$
So, if $\mathcal F$ is an ultrafilter, then $a$ has the same measure in $M^{\mathcal F}$ as in $M$.
This is however not true if we replace $\mathcal F$ with an arbitrary ultracharge.
For example, let $M=\{0,1\}$ where $\mu(0)=\mu(1)=\frac{1}{2}$ and set $N=\frac{1}{2}M+\frac{1}{2}M$.
Then, $M$, $N$ satisfy the condition $\int\rho(x,y)dy=\frac{1}{2}$.
By computing this integral for every $a\in N$, one shows that $\mu(a)=\frac{1}{4}$ for every $a\in N$.
\vspace{3mm}

\noindent{\bf Henkin's construction}\\
Henkin's construction can be adapted to the integral logic context.
Besides the witness property, we need the mean value property which fits well with the argument in affine logic.

\begin{lemma}\label{mean0}
Assume both $T, 0\leqslant\sigma$ and $T,\sigma\leqslant0$ are affinely satisfiable.
Then, $T, \sigma=0$ is affinely satisfiable.
\end{lemma}
\begin{proof}
We may further assume that $T$ is affinely closed. Assume the claim does not hold.
Then, for some $0\leqslant\eta$ in $T$ and $r,s\geqslant0$,
the condition $r\sigma\leqslant s\sigma+\eta$ is unsatisfiable.
So, for some $\epsilon>0$, every model satisfies $(s-r)\sigma+\eta\leqslant-\epsilon$.
Assume $r\leqslant s$. Then, every model satisfies $(s-r)\sigma+\eta\leqslant-\epsilon$.
This is a contradiction. The case $s\leqslant r$ leads to a contradiction similarly.
\end{proof}

\begin{lemma}\label{mean}
Assume both $T,0\leqslant\sup_x\phi(x)$ and\ \ $T,\inf_x\phi(x)\leqslant0$ are affinely satisfiable.
Let $c$ be a new constant symbol. Then, $T,\phi(c)=0$ is affinely satisfiable.
\end{lemma}
\begin{proof}
We may further assume that $T$ is affinely closed. Suppose the claim does not hold.
Then, for some $0\leqslant\eta$ in $T$ and $r,s\geqslant0$,
the condition $r\phi(c)\leqslant s\phi(c)+\eta$ is unsatisfiable.
So, for some $\epsilon>0$, every model satisfies $(s-r)\phi(c)+\eta\leqslant-\epsilon$.
Clearly, $r\neq s$. Assume $r<s$. Then, every model satisfies $$(s-r)\sup_x\phi(x)+\eta\leqslant-\epsilon.$$
This is a contradiction. The case $s<r$ leads to a contradiction similarly.
\end{proof}
\bigskip

\begin{lemma}
Assume $T$ is affinely satisfiable and $\phi(x)$ is formula. Let $c$ be a new constant symbol.
Then, $T,\int\phi(x)dx=\phi(c)$ is affinely satisfiable.
\end{lemma}
\begin{proof}
Set $$\psi(x)=\phi(x)-\int\phi(x)dx.$$
Then, every model satisfies both $\inf_x\psi(x)\leqslant0$ and $0\leqslant\sup_x\psi(x)$.
Hence, they are both affinley satisfiable with $T$. Now, use Lemma \ref{mean}.
\end{proof}

\begin{lemma}
Assume $T$ is affinely satisfiable and $\phi(x)$ is formula. Let $c$ be a new constant symbol.
Then, $T\cup\{\phi(c)\leqslant\inf_x\phi(x)\}$ is affinely satisfiable.
\end{lemma}
\begin{proof}
Suppose not. We may assume $T$ is affinely closed. Then, there are $\sigma\leqslant\eta$ in $T$ and $r,s>0$ such that
$$r\phi(c)+s\sigma\leqslant r\inf_x\phi(x)+s\eta$$ is unsatisfiable. So, for some $\epsilon>0$
$$\vDash r\inf_x\phi(x)+s\eta\leqslant r\phi(c)+s\sigma-\epsilon.$$
So, $$\vDash r\inf_x\phi(x)+s\eta\leqslant r\inf_x\phi(x)+s\sigma-\epsilon$$
which implies that $\vDash \eta\leqslant\sigma-\frac{\epsilon}{s}$. This is a contradiction.
\end{proof}

\begin{definition}\label{witness-mean}
\em{An $L$-theory $T$ has \emph{the witness property} if for each $\phi(x)$ there exists $c$ such that
$\phi(c)\leqslant\inf_x\phi(x)$ belongs to $T$. It has the \emph{mean value property} if for each $\phi(x)$
there exists $c$ such that $\int\phi(x)dx=\phi(c)$ belongs to $T$.}
\end{definition}

The following lemma has a routine proof using infinitary chains of affinely satisfiable theories.

\begin{lemma}
Assume $T$ is an affinely satisfiable theory in $L$. Then there is a language $\bar{L}\supseteq L$
and an affinely satisfiable $\bar{L}$-theory $\bar{T}\supseteq T$ which has the witness as well as
the mean value properties and is maximal with these properties.
\end{lemma}

\begin{proposition}\label{model construction}
Let $T$ be a maximal affinely satisfiable theory in $L$ with the witness and mean value properties. Then $T$ has a model.
\end{proposition}
\begin{proof}
By maximality, for every sentence $\phi$ there is a unique $r$ such that $\phi=r\in T$.
We denote this $r$ by $\phi^T$.
Let $C$ be the set of constant symbols of $L$ and define a pseudometric on $M$
(which we denote it by $\rho$) by setting $\rho(c,e)=(\rho(c,e))^T$.
Let $M$ be the resulting quotient metric space where its metric is denoted by $\rho$ again.
The equivalence class of $c$ is denoted by $\hat c$. Define a metric structure on $M$ by setting:

- $c^M=\hat c$

- $R^M(\hat a_1,...,\hat a_n)=(R(a_1,...,a_n))^T$

- $F^M(\hat a_1,...,\hat a_n)=\hat b$ \ \ if \ \ $\rho(F(a_1,...,a_n),b)=0\in T$.
\vspace{1mm}

Note that $R^M$ and $F^M$ are well-defined. For every $L$-formula $\phi(x)$,
consider the function $\phi^T:M\rightarrow\Rn$ defined by $\hat c\mapsto(\phi(c))^T$.
The set of such functions is a majorizing subset of $\mathbf{C}_b(M)$.
It is clear that the functional $$\Lambda:\phi^T(x)\mapsto\Big(\int\phi(x)dx\Big)^T$$ is linear.
On the other hand, if $0\leqslant\phi^T(x)$ as a function, then $0\leqslant\phi(c)\in T$ for every $c$.
Hence, by the mean value property, for some $c$ the conditions $$0\leqslant\phi(c)=\int\phi(x)dx$$ belong to $T$.
This shows that $\Lambda$ is positive.
Now, using Corollary \ref{KR}, we obtain a charged $L$-structure $(M,\mu)$ such that for every $\phi(x)$
$$\int\phi^T(x)d\mu=\Big(\int\phi(x)dx\Big)^T.$$
Finally, we must prove that for every $L$-formula $\phi(\x)$ and $a_1,...,a_n\in C$,
$$\phi^M(\hat a_1,...,\hat a_n)=(\phi(a_1,...,a_n))^T.$$
We consider the integral case. Assume the claim is proved for $\phi(x)$ so that
for every $a\in C$, $\phi^M(\hat a)=(\phi(a))^T$.
Then, by the construction of $\mu$ and the induction hypothesis,
$$\Big(\int\phi(x)dx\Big)^T=\int\phi^T(x)d\mu=\int\phi^M(x)d\mu.$$
To complete the proof, use Proposition \ref{exist1} to find a complete model of $T$. 
\end{proof}

\begin{corollary}
Every affinely satisfiable theory is satisfiable.
\end{corollary}

An easy consequence of the affine compactness is that if both $\sigma\leqslant0$ and
$0\leqslant\sigma$ have model, then $\sigma=0$ has a model.
One can also verify that if $T\vDash0\leqslant\phi$ then for each $\epsilon>0$
there exists a finite $\Delta\subseteq T$ such that $\Delta\vDash-\epsilon\leqslant\phi $.
A more important consequence is the following.

\begin{corollary}
If $M\equiv N$, then  there exists $K$ such that $M\preccurlyeq K$ and $N\preccurlyeq K$.
\end{corollary}

The elementary amalgamation property is proved in a similar way.
In fact, many basic results in affine logic hold in the framework of AL$_{\int}$ too.
The proofs are often similar and rely on the affine compactness theorem and the properties of quantifiers $\sup$-$\inf$.

\section{Completeness}
By $\phi=\psi$ is meant $\{\phi\leqslant\psi, \psi\leqslant\phi\}$.
We may however use it as a primitive expression. By $\phi[t/x]$ is meant the result
of substituting the term $t$ in place of the free occurrences of $x$ in $\phi$.
The notion of substitutability of $t$ for $x$ is defined as in first order logic.
\vspace{1mm}

{\bf Linearity axioms}:

(A1) $r\leqslant s$ \hspace{13mm} if $\Rn\vDash r\leqslant s$

(A2) $\phi+(\psi+\theta)=(\phi+\psi)+\theta$

(A3) $\phi+\psi=\psi+\phi$

(A4) $0+\phi=\phi$

(A5) $r(\phi+\psi)=r\phi+r\psi$

(A6) $(r+s)\phi=r\phi+s\phi$

(A7) $r(s\phi)=(rs)\phi$

(A8) $1\phi=\phi$

(A9) $0\phi=0$
\vspace{1mm}

{\bf Quantification axioms}:

(A10) $\phi[t/x]\leqslant(\sup_x\phi)$ \hspace{20mm} (if $t$ is substitutable for $x$)

(A11) $\sup_x(\phi+\psi)=\sup_x\phi(x)+\psi$ \hspace{6mm} ($x$ is not free in $\psi$)

(A12) $\sup_x(\phi+\psi)\leqslant\sup_x\phi+\sup_x\psi$

(A13) $\sup_x(r\phi)=r\sup_x\phi$\ \ \ \ where $r\geqslant 0$

(A14) $\sup_x\phi=-\inf_x-\phi$

(A15) $\int 1 dx=1$

(A16) $\int(\phi+\psi)dx=\int\phi dx+\int\psi dx$

(A17) $\int r\phi dx=r\int\phi dx$

(A18) $\int\phi dx=\phi$ \ \ \ \ if $x$ is not free in $\phi$
\vspace{1mm}

{\bf Pseudometric axioms}:

(A19)  $\rho(x,x)=0$

(A20) $\rho(x,y)=\rho(y,x)$

(A21) $\rho(x,z)\leqslant \rho(x,y)+\rho(y,z)$
\vspace{1mm}

{\bf Bound and Lipschitz axioms}:

(A22) $\rho(F(\bar{x}),F(\bar{y}))\leqslant\lambda_F\ \rho(\bar{x},\bar{y})$ \hspace{10mm} ($F\in L$)

(A23) $R(\bar{x})-R(\bar{y})\leqslant\lambda_R\ \rho(\bar{x},\bar{y})$ \hspace{12mm} ($R\in L$)

(A24) $0\leqslant R(\x)\leqslant1$ \hspace{26mm} $(R\in L$ or $R=\rho$)
\vspace{1mm}

{\bf Deduction rules}: \vspace{1mm}

(R1) $\frac{\phi\leqslant\psi,\ \psi\leqslant\theta}{\phi\leqslant\theta}$ \vspace{2mm}

(R2) $\frac{\phi\leqslant\psi}{\phi+\theta\leqslant\psi+\theta}$ \vspace{2mm}

(R3) $\frac{0\leqslant r,\ \phi\leqslant\psi}{r\phi\leqslant r\psi}$ \vspace{2mm}

(R4) $\frac{\phi\leqslant\psi}{\sup_x\phi\leqslant\sup_x\psi}$ \vspace{2mm}

(R5) $\frac{\phi\leqslant\psi}{\int\phi dx\leqslant\int\phi dx}$ \vspace{2mm}
\bigskip

Below, $\Sc,\Sc_1,...$ denote conditions.
We often write $\Gamma,\Sc$ for $\Gamma\cup\{\Sc\}$. We also write $\Gamma\vdash\Sc_1,\Sc_2$
if $\Gamma\vdash\Sc_1$ and $\Gamma\vdash\Sc_2$.
We define an increasing chain of length $\omega$ of relations $\vdash_n$.

\begin{definition} \label{proofdfn}
{\em We write $\Gamma\vdash_0\Sc$ if $\Sc$ belongs to $\Gamma$ or is a logical axiom.
Suppose $\vdash_k$ has been defined for $k<n$.
Then, we write $\Gamma\vdash_n\Sc$ if any one of the following requirements is satisfied:
\begin{quote}

- there is $k<n$ such that $\Gamma\vdash_k\Sc$

- there are $k<n$ and instance $\frac{\Sc_1,\ \Sc_2}{\Sc}$ of the logical rules (R1-R3)
such that $\Gamma\vdash_k\Sc_1,\Sc_2$

- $\Sc$ is the condition $\sup_x\phi\leqslant\sup_x\psi$,\ \ $x$ is not free
in $\Gamma$ and $\Gamma\vdash_k\phi\leqslant\psi$ for some $k<n$

- $\Sc$ is the condition $\int\phi dx\leqslant\int\psi dx$,\ \ $x$ is not free
in $\Gamma$ and $\Gamma\vdash_k\phi\leqslant\psi$ for some $k<n$

\end{quote}}
\end{definition}

\begin{definition}
{\em  $\Sc$ is \emph{provable}\index{provable} 
from $\Gamma$, denoted by $\Gamma\vdash\Sc$, if $\Gamma\vdash_n\Sc$ for some $n<\omega$.
A set of conditions $\Gamma$ is \emph{inconsistent} if $\Gamma\vdash 1\leqslant 0$.
Otherwise, it is \emph{consistent}\index{consistent}.}
\end{definition}

One proves by induction on the length of proof that:

\begin{theorem} \label{soundness} \emph{(Soundness)}
If $\Gamma\vdash\Sc$, then $\Gamma\vDash\Sc$. If $\Gamma$ is satisfiable, then $\Gamma$ is consistent.
\end{theorem}

It is easily proved by induction on the complexity of terms and formulas that for every term
$t(\x)$ and formula $\phi(\x)$ one has that
$$\vdash\ -{\sf b}_\phi\leqslant\phi\leqslant{\sf b}_\phi$$
$$\vdash\ \rho(t(\x),t(\y))\leqslant\lambda_t \rho(\x,\y)$$
$$\vdash\ \phi(\x)-\phi(\y)\leqslant\lambda_\phi \rho(\x,\y).$$

The following two lemmas are also routine.
\begin{lemma} \label{easy consequence}
For each $\Gamma$, $\phi,\psi$ and $r,s$ the following hold:

{\em(i)} $r=0\vdash r\phi=0$

{\em(ii)} $\phi=0\vdash r\phi=0$

{\em(iii)} $\{r=s,\ \phi=\psi\}\vdash r\phi=s\psi$

{\em(iv)} If $\Gamma\vdash 1\leqslant 0$ then  $\Gamma\vdash\Sc$
for every condition $\Sc$.
\end{lemma}

\begin{lemma} \label{deduction1}
(i) If $\Gamma\subseteq\Gamma'$ and $\Gamma\vdash_n\Sc$ then $\Gamma'\vdash_n\Sc$.

(ii) Let $\bar\Gamma=\{\Sc:\ \Gamma\vdash\Sc\}$. Then, $\bar\Gamma\vdash\Sc$ implies $\Gamma\vdash\Sc$.

(iii) If $\Gamma\vdash\Sc$, then $\Gamma_0\vdash\Sc$ for some finite $\Gamma_0\subseteq\Gamma$.
\end{lemma}
\begin{proof}
(i): Proceed by induction on $n$. In the quantifier cases a change of variables is needed.

(ii): It is proved by induction on $n$ that $\bar\Gamma\vdash_n\Sc$ implies $\Gamma\vdash\Sc$ for every $\Gamma$ and $\Sc$.

(iii): Obvious.
\end{proof}

\begin{lemma}
$\Gamma,0\leqslant\theta\vdash\phi\leqslant\psi$ if and only if $\Gamma\vdash r\theta+\phi\leqslant\psi$ for some $r\geqslant0$.
\end{lemma}
\begin{proof}
It is proved by induction on $n$ that if $\Gamma,0\leqslant\theta\vdash_n\phi\leqslant\psi$
then there is $r\geqslant0$ such that $\Gamma\vdash r\theta+\phi\leqslant \psi$.
In case $n=0$, set $r=0$ if $\phi\leqslant\psi$ is a logical axiom or belongs to $\Gamma$ and set $r=1$ if
$\phi\leqslant\psi$ coincides with the condition $0\leqslant\theta$.
Assume $$\Gamma,0\leqslant\theta\vdash_n\int\phi dx\leqslant\int\psi dx$$
and for some $k<n$ one has that
$$\ \Gamma,0\leqslant\theta\vdash_k\phi\leqslant\psi\hspace{10mm} (x\ \textrm{is not free in}\ \Gamma,0\leqslant\theta).$$
Then, by the induction hypothesis, for some $r\geqslant0$ one has that
$$\Gamma\ \vdash\ r\theta+\phi\leqslant\psi.$$
Therefore, since $\int\theta dx=\theta$, $$\Gamma\ \vdash\ r\theta+\int\phi dx\leqslant\int\psi dx.$$
Similarly, assume $\Gamma,0\leqslant\theta\ \vdash_n s\phi\leqslant s\psi$ and for some $k<n$ one has that
$$\Gamma,0\leqslant\theta\ \vdash_k\ 0\leqslant s,\hspace{14mm}\Gamma,0\leqslant\theta\ \vdash_k\ \phi\leqslant\psi.$$
Then, there is $r\geqslant0$ such that $\Gamma\vdash r\theta+\phi\leqslant\psi$. Therefore, if $s\geqslant0$,
then $\Gamma\vdash rs\theta+s\phi\leqslant s\psi$. Also, if $s<0$, then we must have that
$\Gamma\vdash\theta\leqslant0$ in which case $\Gamma\vdash\theta+\phi\leqslant\psi$.
All other cases are proved similarly. The inverse direction is obvious.
\end{proof}

\begin{lemma}\label{consequence2}
(i) If $\Gamma$ is consistent, then either $\Gamma,0\leqslant\phi$ is consistent or
there is $\epsilon>0$ such that $\Gamma\vdash\phi\leqslant-\epsilon$
(in which case $\Gamma,\phi\leqslant-\epsilon$ is consistent).

(ii) If $\Gamma\nvdash\phi\leqslant0$ then $\Gamma, 0\leqslant\phi$ is consistent.

(iii) If $\Gamma,-\frac{1}{n}\leqslant\phi$ is consistent for each $n\geqslant1$, then $\Gamma,0\leqslant\phi$ is consistent.
\end{lemma}
\begin{proof}
(i) Assume $\Gamma,0\leqslant\phi\vdash1\leqslant0$. Then, there is $r\geqslant0$ such that
$\Gamma\vdash r\phi+1\leqslant0$. Clearly, $r\neq0$. So, $\Gamma\vdash\phi\leqslant\frac{-1}{r}$.

(ii) is a consequences of (i).

(iii): Otherwise, there must exist $\epsilon>0$ such that $\Gamma\vdash\phi\leqslant-\epsilon$.
This is a contradiction.
\end{proof}

\begin{lemma}
Assume $\Gamma,0\leqslant\theta\ \vdash\phi\leqslant\psi$ and $\Gamma,\theta\leqslant0\ \vdash\phi\leqslant\psi$.
Then, $\Gamma\vdash\phi\leqslant\psi$.
\end{lemma}
\begin{proof}
There are $r,s\geqslant0$ such that
$$\Gamma\ \vdash\ r\theta+\phi\leqslant\psi,\ \ \ \ \Gamma \vdash\ -s\theta+\phi\leqslant\psi.$$
If $r=0$ or $s=0$, we are done. Otherwise, we have that
$$\Gamma\ \vdash\ (r+s)\phi\leqslant(r+s)\psi$$ which implies the intended claim.
\end{proof}

\begin{lemma}\label{mean-proof0}
Assume both $\Gamma,0\leqslant\phi$ and $\Gamma,\phi\leqslant0$ are consistent. Then, $\Gamma,\phi=0$ is consistent.
\end{lemma}
\begin{proof}
Suppose that $$\Gamma,\phi\leqslant0,\ 0\leqslant\phi\ \ \vdash\ 1\leqslant0.$$
Then, there is $r>0$ such that
$$\Gamma,\phi\leqslant0\ \ \vdash\ \ r\phi+1\leqslant0\ \ \vdash\ \ \phi\leqslant-\frac{1}{r}.$$
Hence, there is $s\geqslant0$ such that
$$\Gamma\ \vdash\ -s\phi+\phi\leqslant-\frac{1}{r}$$
Clearly, $s\neq1$. Moreover, in the case $s<1$ we have that $\Gamma\vdash\phi\leqslant\frac{-1}{r(1-s)}$
and in the case $s>1$ we have that $\Gamma\vdash\frac{1}{r(s-1)}\leqslant\phi$.
Both these consequences contradict the assumptions.
\end{proof}

\begin{lemma}\label{c1} Assume $y$ does not occur in $\Sc$ and it is not free in $\Gamma$.
If $c$ does not occur in $\Gamma$ and  $\Gamma\vdash\Sc$, then $\Gamma\vdash\Sc[y/c]$.
\end{lemma}
\begin{proof} We prove the claim by induction on $\Gamma\vdash_n\Sc$.
The case $n=0$ is obvious. Suppose it has been proved for $k<n$. We have the following cases:
\vspace{1mm}

- If $\Gamma\vdash_k\Sc$ for some $k<n$, we are done.
\vspace{1mm}

- There are $\Sc_1, \Sc_2$ such that
$\Gamma\vdash_k\Sc_1, \Sc_2$ for some $k<n$ and $\frac{\Sc_1, \Sc_2}{\Sc}$ is an instance of
the rules (R1-R3). Then, by the induction hypothesis,
$$\Gamma\ \vdash\ \ \Sc_1[y/c], \Sc_2[y/c]\ \ \vdash\ \ \Sc[y/c].$$
So, $\Gamma\ \vdash\ \Sc[y/c]$.
\vspace{1mm}

- $\Sc$ is of the form $\sup_x\phi\leqslant\sup_x\psi$ or of the form $\int\phi dx\leqslant\psi dx$, \
$x$ is not free in $\Gamma$ and there exists $k<n$ such that
$$\Gamma\vdash_k\phi\leqslant\psi.$$ Then, by induction $$\Gamma\vdash\phi[y/c]\leqslant\psi[y/c].$$
Therefore, since $y\neq x$, we have that $\Gamma\vdash\Sc[y/c]$.
\end{proof}

\begin{corollary}
Let $\Gamma$ be a set of closed conditions in which $c$ does not occur.
If $y$ does not occur in $\phi$ and $\Gamma\vdash\phi\leqslant r$ then
$$\Gamma\ \vdash\ \sup_y\phi[y/c]\leqslant r,\ \ \ \ \ \ \ \Gamma\ \vdash\ \int\phi[y/c]dy\leqslant r.$$
\end{corollary}
\vspace{2mm}

By Zorn's lemma, every consistent theory is contained in a maximally consistent one.
Every maximally consistent theory $\Gamma$ is closed, i.e. $\Gamma\vdash\Sc$ implies that $\Sc\in\Gamma$.
Moreover, for every sentences $\phi,\psi$ either $\phi\leqslant\psi\in\Gamma$ or $\psi\leqslant\phi\in\Gamma$.

\begin{lemma}\label{equality} Let $\Gamma$ be a maximal consistent theory and $\phi$ be a sentence.
Then there is a unique $r$ such that $\phi=r\in \Gamma$.
\end{lemma}
\begin{proof} Let
$$r=\sup{\{r': r'\leqslant\phi\in \Gamma}\},\hspace{12mm} s=\inf{\{s': \phi\leqslant s'\in \Gamma}\}.$$
For every $\epsilon>0$, both $r-\epsilon\leqslant\phi$ and $\phi\leqslant s+\epsilon$
belong to $\Gamma$. Therefore, $r\leqslant s$.
Also, for each $\epsilon>0$ we have that $(r+\epsilon)\leqslant\phi\notin\Gamma$.
So, $\phi\leqslant r+\epsilon\in \Gamma$ and hence $s\leqslant r+\epsilon$. We conclude that $r=s$.
Note that $r\leqslant\phi\in\Gamma$ since $\Gamma,r\leqslant\phi$ is consistent.
Similarly, $\phi\leqslant r\in \Gamma$.
\end{proof}

\begin{definition}
\em{An $L$-theory $\Gamma$ has \emph{the witness property} if for each $\phi(x)$ there is $c$ such that
$\sup_x\phi(x)\leqslant\phi(c)$ belongs to $\Gamma$. It has the \emph{mean value property}
if for each $\phi(x)$ there is $c$ such that $\int\phi(x)dx=\phi(c)$ belongs to $\Gamma$.}
\end{definition}

\begin{corollary}
Assume $\Gamma$ is consistent and $\phi(x)$ is a formula. Let $c$ be a new constant symbol. Then,

(i) $\Gamma,\int\phi(x)dx=\phi(c)$ is consistent

(ii) $\Gamma,\sup_x\phi(x)\leqslant\phi(c)$ is consistent.
\end{corollary}
\begin{proof}
(i) We claim that $\Gamma,\phi(c)\leqslant\int\phi(x)dx$ is consistent.
Assume not. Then, for some $\epsilon>0$, $$\Gamma\ \vdash\ \int\phi(x)dx\leqslant\phi(c)-\epsilon.$$
Therefore, $$\hspace{10mm}\Gamma\ \vdash\ \int\phi(x)dx\leqslant\int\phi(x)dx-\epsilon.$$
This is a contradiction. It is similarly proved that $\Gamma,\int\phi(x)dx\leqslant\phi(c)$
is consistent. We conclude by \ref{mean-proof0} that
$\Gamma,\int\phi(x)dx=\phi(c)$ is consistent.

(ii) Similar.
\end{proof}

The following proposition is proved using chain arguments.

\begin{proposition} \label{witness}
Let $\Gamma$ be a consistent $L$-theory. Then, there are $\bar L\supseteq L$ and a maximal
consistent $\bar L$-theory $\bar{\Gamma}\supseteq \Gamma$ having the witness property and the mean value property.
\end{proposition}

The following proposition is then proved similar to Proposition \ref{model construction}.

\begin{proposition}\label{completeness}
Let $T$ be a maximal consistent theory in $L$ with the witness and mean value properties.
Then $T$ has a model.
\end{proposition}

\begin{theorem}{\em(Weak completeness)} Every consistent theory is satisfiable.
\end{theorem}

\begin{theorem} {\em (Approximate strong completeness)}
$\Gamma\vDash0\leqslant\phi$ if and only if $\Gamma\vdash-\frac{1}{k}\leqslant\phi$ for all $k\geqslant1$.
\end{theorem}
\begin{proof}
Suppose that $\Gamma\nvdash-\frac{1}{k}\leqslant\phi$ for some $k\geqslant1$.
Then, $\Gamma,\phi\leqslant-\frac{1}{k}$ is consistent. Let $M$ be a model of this theory.
Then, we must have that $0\leqslant\phi^M\leqslant-\frac{1}{k}$.
\end{proof}

The proof system given above is not sufficiently strong to obtain an exact form of strong completeness theorem,
i.e. to deduce $\Gamma\vdash\Sc$ from $\Gamma\vDash\Sc$. To obtain such a result, we need
an infinitary proof system, i.e. $\vdash_\alpha$ for every $\alpha<\omega_1$.

\section{Types} \label{Types}
Let $T$ be a (affinely) complete $L$-theory. For $|\x|=n\geqslant1$, \ $D_{n}(T)$ denotes the vector space
of formulas with free variables $\x$ modulo $T$-equivalence.
This is a partially ordered normed vector space by
$$\|\phi\|=\sup_{\a\in M}|\phi^M(\a)|$$
$$0\leqslant\phi\ \ \ \ \mbox{if} \ \ \ \ T\vDash0\leqslant\phi(\x).$$
A \emph{$n$-type} is a positive linear map $p: D_n(T)\rightarrow\Rn$ such that $p(1)=1$.
Note that $p$ extends uniquely to a positive linear map on the completion of $D_n(T)$.
The set of $n$-types of $T$ is denoted by $K_n(T)$.
Types over parameters from $A\subseteq M\vDash T$ are defined similarly. $K_n(A)$ denotes the set of $n$-types over $A$.
By the Banach-Alaoglu theorem, $K_n(T)$ is a compact convex subset of the unit ball of
$( D_n(T))^*$ (similarly for $K_n(A)$).
The logic topology of $K_n(T)$ is generated by sets of the form $\{p: 0<p(\phi)\}$ where $\phi(\x)$ is a formula.
A type $p(\x)$ is realized by $\a\in M$ if $p(\phi)=\phi^M(\a)$ for every $\phi$.
Integration is itself a type on $M$.


\begin{lemma}
\label{realize}Let $A\subseteq M$ and $p(\x)$ be a type over $A\subseteq M$. Then $p$, is realized in some $M\preccurlyeq N$.
\end{lemma}

Using Lemma \ref{realize} and Proposition \ref{downward2}, one can prove many results concerning saturated and homogeneous models.
We just mention some examples whose proofs are similar to the classical case (see also \cite{ACMT}).

\begin{definition}
\emph{Let $\kappa$ be an infinite cardinal and $M\vDash T$.\\
\emph{(i)} $M$ is $\kappa$-\emph{saturated}\index{saturated} if for each $A\subseteq M$
with $|A|<\kappa$, every type $p\in K_1(A)$ (and hence every $p\in K_n(A)$) is realized in $M$.\\
\emph{(ii)} $M$ is $\kappa$-\emph{homogeneous}\index{homogeneous} if for every $\a,\b,c\in M$
with $|\a|=|\b|<\kappa$ and $\a\equiv\b$, there exists $e\in M$ such that $\a c\equiv\b e$.
It is \emph{strongly $\kappa$-homogeneous}\index{strongly homogeneous} if for $\a,\b$ as above there is an automorphism $f$ of $M$ such that $f(\a)=\b$.\\
\emph{(iii)} $M$ is $\kappa$-\emph{universal} if every $N\equiv M$ with $\mathsf{dc}(N)<\kappa$ is elementarily embedded in $M$.}
\end{definition}

\begin{proposition}
If $M\equiv N$ are $\kappa$-saturated where $\kappa=\mathsf{dc}(M)=\mathsf{dc}(N)$, then $M\simeq N$.
\end{proposition}

\begin{proposition}
$M$ is $\kappa$-saturated if and only if it is $\kappa$-homogeneous and $\kappa^+$-universal.
\end{proposition}

\begin{proposition} \label{homogeneous2}
Assume $M\vDash T$ and $\mathsf{dc}(M)\leqslant2^{\kappa}$.
Then, there is a $\kappa$-saturated strongly $\kappa$-homogeneous $M\preccurlyeq N$ such that $|N|\leqslant2^{\kappa}$.
\end{proposition}

\begin{proposition}\label{homogeneous}
Let $M,N\vDash T$ be separable and $\aleph_0$-homogeneous.
Then $M\simeq N$ if and only if they realize the same types in every $K_n(T)$.
\end{proposition}

Let $M\vDash T$ be $\aleph_0$-saturated. Then $$\varrho(p,q)=\inf\{\rho(\a,\b):\ \a\vDash p,\ \b\vDash q\}$$
defines a metric on $K_n(T)$. As in CL, it is proved that $(K_n(T),\tau,\varrho)$ is a topometric space.
We can also define a probability measure on $M^n$ and then transfer it to $K_n(T)$.
For this purpose, let $\x=x_1\cdots x_n$ and define a map $\Lambda$ by setting
$$\phi(\x,\a)\mapsto\int\cdots\int\phi^M(\x,\a)\ dx_1\cdots dx_n$$
where $\phi(\x,\z)$ is any $L$-formula and $\a\in M$.
By Corollary \ref{KR}, there is a regular Borel charge $\mu_n$ on $M^n$ such that for every $\phi(\x,\a)$,
$$\Lambda(\phi)=\int\phi^M\ d\mu_n.$$
$T$ is a \emph{Fubini theory} if for every $\phi(x,y,\z)$ one has that $$\iint\phi dxdy\equiv_T\iint\phi dydx.$$
For example, if $M$ is compact, its affine theory is Fubini.
If $T$ is Fubini, $\mu_n$ defined above is independent of the order of variables.
We can now define a Borel charge on $K_n(T)$. Assume further that $M$ be $\aleph_0$-saturated.
Then the map $$tp:M^n\rightarrow K_n(T),\ \ \ \ \ \a\mapsto tp(\a)$$
is surjective and logic-continuous (as well as $1$-Lipschitz with respect to the metric topology of $K_n(T)$).
For every formula $\phi(\x)$, let $\hat{\phi}:K_n(T)\rightarrow\Rn$ be defined by $\hat\phi(p)=p(\phi)$.
Also define $$\bar\Lambda(\hat{\phi})=\Lambda(\phi)=\int\phi\ d\mu_n.$$
Again, by Corollary \ref{KR}, there is a Borel charge $\hat\mu_n$ on $K_n(T)$
such that for every $\phi$ $$\bar\Lambda(\phi)=\int\phi\ d\hat\mu_n.$$
Briefly, for every $M\vDash T$, we have Borel charges $\mu_n$ on $M^n$ and (if it is $\aleph_0$-saturated)
a Borel charge $\hat\mu_n$ on $K_n(T)$ such that for every formula $\phi(x_1,...,x_n)$
$$\int\phi^M\ d\mu_n=\int\cdots\int\phi^M\ dx_1\cdots dx_n=\int\hat\phi\ d\hat\mu_n.$$
Note that $\hat\mu_n$ does not depend on the choice of $M$.
Finally, since $K_n(T)$ is compact with respect to the logic topology, $\mu_n$ extends uniquely to a
$\sigma$-additive probability measure on $K_n(T)$.

We can also embed $D_n(K)$ in $D_{n+1}(T)$ in a natural way. For each $n$, the integration operator defines a positive linear map
$$D_{n+1}(T)\rightarrow D_n(T),\ \ \ \ \ \ \ \ \ \phi(\x,y)\mapsto\int\phi(\x,y)dy.$$
This induces an affine map $+:K_{n}(T)\rightarrow K_{n+1}(T)$ by
$$p^+(\phi(\x,y))=p(\int\phi(\x,y)dy).$$
This is an injective logic continuous map. Recall again that by the affine compactness, if
$\inf_x\phi^M(x)\leqslant r\leqslant\sup_x\phi^M(x)$, then $\phi(x)=r$ is satisfiable in
an elementary extension of $M$. This is the mean value property for $\phi(x)$.

\begin{lemma}
Every map $p\mapsto p^+$ are metric preserving.
\end{lemma}
\begin{proof}
Assume $\a c\vDash p^+$ and $\b e\vDash q^+$ where $|\a|=|\b|=n$.
Then, $\a\vDash p$ and $\b\vDash q$.
So, $$\mathbf{\varrho}(p,q)\leqslant\varrho(p^+,q^+).$$
For the inverse inequality, assume $\a\vDash p$ and $\b\vDash q$ such that $\varrho(p,q)=\rho(\a,\b)$.
We only need to prove that there exists $c\in M$ such that for every $\phi(\x,y)$
$$p^+(\phi)=\phi^M(\a,c),\ \ \ \ \ q^+(\phi)=\phi^M(\b,c).$$
This means that for every $\phi(\x,y)$$$\phi^M(\a,c)=\int\phi^M(\a,y)dy,\ \ \ \ \ \ \phi^M(\b,c)=\int\phi^M(\b,y)dy.$$
For this purpose, by affine compactness, we must prove that for every $\phi(\x,y),\psi(\x,y)$ there exists $c$ satisfying the condition $$\phi^M(\a,y)+\psi^M(\b,y)=\int\phi^M(\a,y)dy+\int\psi^M(\b,y)dy.$$
Let $\theta(y)=\phi(\a,y)+\psi(\b,y)$.
So, we need to find $c$ such that $\theta(c)=\int\theta(y)dy$. This is guaranteed by the mean value property for $\theta$.
\end{proof}

On the other hand, the restriction of $p^+(\x,y)$ to $\x$ is the type $p(\x)$.
We may therefore regard $p\mapsto p^+$ as the inclusion map and write
$$K_1(T)\subseteq K_2(T)\subseteq\cdots$$ where the inclusions are metric preserving and logic continuous.

\section{The isomorphism theorem}
As in AL, to prove the isomorphism theorem, we need to generalize the powermean construction.
Let $L$ be a metric language (without the charge symbol) and $M$ be a structure in $L$. Let $(I,\mathcal A,\wp)$ be a charge space.
A map $a:I\rightarrow M$ is called measurable if $a^{-1}(B)\in\mathcal A$ for every Borel $B\subseteq M$.
A function $u:I\rightarrow\Rn$ is measurable if $u^{-1}(X)\in\mathcal A$ for every $X\in\mathcal H(\Rn)$.
The topological weight of $M$ is defined by
$$w(M)=\min\{|\mathcal{O}|:\ \mathcal{O}\ \mbox{is\ a\ basis\ for } M \}+\aleph_0.$$
A $L$-structure $M$ is \emph{$\mathcal A$-meanable}\index{meanable} if $M$ is finite or $\mathcal{A}$ is $w(M)^+$-complete.
Assume $M$ is $\mathcal A$-meanable.
Then, for every bounded continuous $f:M^n\rightarrow\Rn$ and measurable maps $a^1,...,a^n:I\rightarrow M$,
the function $u(i)=f(a^1_i,...,a^n_i)$ is measurable (hence integrable).
Similarly, for every continuous $g:M^n\rightarrow M$,\ \ $a(i)=g(a^1_i,...,a^n_i)$ is measurable.
We can now define the powermean $M^\wp$.

For measurable $a,b:I\rightarrow M$ let $a\sim b$ if $\int d(a_i,b_i)d\wp=0$. The equivalence class of $a=(a_i)$ is denoted by $[a_i]$.
Let $M^{\wp}$ be the set of equivalence classes of measurable maps $a:I\rightarrow M$.
We can then define a metric structure on $M^{\wp}$ exactly as in section \ref{section compactness}.
We have then the following results (see \cite{ACMT}, \cite{Isomorphism}).

\begin{theorem} \label{powermean} \emph{(Powermean)}
Let $(I,\mathcal A,\wp)$ be a charge space and $M$ be an $\mathcal A $-meanable structure.
Then, for each $L$-formula $\phi(\x)$ and $[a^1_i],...,[a^n_i]\in M^{\wp}$,
$$\phi^{M^{\wp}}([a^1_i],...,[a^n_i])=\int_I\phi^M(a^1_i,...,a^n_i)d\wp.$$
\end{theorem}

\begin{theorem} \label{isomorphism}
Assume $M\equiv N$. Then, there is a charge space $(I,\mathcal A,\wp)$, such $M$ and $N$
are $\mathcal A$-meanable and $M^{\wp}\simeq N^{\wp}$.
\end{theorem}

Now, add the charge symbol to $L$ and denote it by $L_\mu$.
We can then extend the isomorphism theorem to probability structures.

\begin{theorem}\label{isomorphism charged}
For $L_\mu$-structures $M,N$, if $M\equiv N$, then $M^\wp\simeq N^\wp$ for some $\wp$.
\end{theorem}
\begin{proof}
For each $L_\mu$-formula $\phi(\x)$ with $|\x|=n\geqslant1$, take a $n$-ary relation symbol $R_\phi$
and set $$\bar L=L\cup\{R_\phi:\ \phi\ \textrm{is a $L_\mu$-formula}\}.$$
Let $\bar M$ be the (measure-free) expansion of $M$ to the language $\bar L$ obtained by defining $R^M_\phi(\x)=\phi^M(\x)$ for every $\phi$.
Define $\bar N$ in a similar way. Also, set ${\sf M}=(\bar M,\mu_M)$ and ${\sf N}=(\bar N,\mu_N)$.
Then, ${\sf M}\equiv{\sf N}$ in the language $\bar{L}_\mu$. Therefore, $\bar M\equiv\bar N$ in the language $\bar L$.
By Theorem \ref{isomorphism}, there is a charge space $\wp$ on some index set $I$ such that
$\bar M^\wp\simeq\bar N^\wp$. Now, we have only to prove that $M^\wp\simeq N^\wp$ in the language $L_\mu$.
To see this, take any formula $\phi(\x)$ in $L_\mu$ and $\a\in M^\wp$.
Note that $R_\phi(\x)=\phi(\x)$ holds in both $\sf M^\wp$ and $\sf N^\wp$. Then, we have that
$$\phi^{M^\wp}(\a)=\phi^{{\sf M}^\wp}(\a)=R^{{\sf M}^\wp}_\phi(\a)=R^{{\bar M}^\wp}_\phi(\a)$$
$$=R^{{\bar N}^\wp}_\phi(f(\a))=R^{{\sf N}^\wp}_\phi(f(\a))=\phi^{{\sf N}^\wp}(f(\a))=\phi^{N^\wp}(f(\a)).$$
\end{proof}

Using Theorem \ref{isomorphism charged}, an AL$_{\int}$ variant of the Robinson consistency theorem is proved.

\begin{proposition}
Let $T_i$ be a satisfiable theory in $L_i$ for $i=1,2$ and $T=T_1\cap T_2$ be complete in $L=L_1\cap L_2$.
Then, $T_1\cup T_1$ is satisfiable.
\end{proposition}

An appropriate interpolation theorem can be then deduced from the above result as in AL (see \cite{ACMT}).

Proposition \ref{isomorphism charged} helps us to give a characterization of the notion of non-forking extension
for types in AL (see \cite{Ibarlucia2} for definitions) in terms of AL$_{\int}$-embeddings.
Let $T$ be a complete stable affine theory in the metric language $L$ (without the charge symbol) and $M\vDash T$.
A relation $R:M^n\rightarrow\Rn$ is $M$-definable if it is the uniform limit of a
sequence of formulas $\theta_k^M(\y)$ with parameters in $M$.
A \emph{definition scheme} for $p(x)\in K_1(M)$ is a map which assigns to each formula $\phi(x,\y)$ (with parameters in $M$) a
$M$-definable relation $\frak d\phi(\y)$ such that $$\frak d\phi(\a)=p(\phi(x,\a))\hspace{14mm} \forall\a\in M.$$
As stated before, $p$ extends to the completion of $D_1(T)$ in a continuous way. Correspondingly,
$\frak d$ extends to the family of all $M$-definable relations, i.e. for each $R(x,\y)$, 
there is a $M$-definable $\frak dR(\y)$ such that $$\frak dR(\a)=p(R(x,\a))\hspace{14mm}\a\in M.$$
By stability, every such type has a definition scheme.
By Proposition \ref{KR}, there is a Borel charge $\mu_p$ on $M$ such that $p(\phi(x))=\int\phi^M(x)d\mu_p$
for every $\phi(x)$ with parameters in $M$. Let $L_\mu$ be the language $L$ augmented with the charge symbol $\mu$.
Then, $(M,\mu_p)$ is a structure in $L_\mu$ and
for each $\phi(x,\y)$ $$\frak d\phi(\a)=p(\phi(x,\a))=\int\phi(x,\a)d\mu_p\hspace{14mm} \forall\a\in M.$$
Note that different extensions $\mu_p$ of $p$ give rise to the same probability structure up to isomorphism.
Let $(I,\mathcal A,\wp)$ be a charge space for which $M$ is $\mathcal A$-meanable.
Define a type $p^{\wp}\in K_1(M^\wp)$ by $$p^{\wp}(\phi(x,\a))=\int_I p(\phi(x,\a_i))d\wp\hspace{12mm} (\a\in M^\wp).$$
We have therefore that $$p^{\wp}(\phi(x,\a))=\int_I\frak d\phi(\a_i)d\wp=(\frak d\phi)^{M^\wp}(\a).$$
Therefore, $p^{\wp}$ is again definable over $M$ and hence non-forking over $p$.

\begin{proposition}\label{def-scheme1}
Let $T$ be a stable AL-theory, $M\vDash T$ and $p(x)\in K_1(M)$.

(i) In $(M,\mu_p)$, every $L_\mu$-formula is equivalent to a $L(M)$-definable relation.

(ii) If $M\preccurlyeq N$, then $q\in K_1(N)$ is a non-forking extension of $p$ iff $(M,\mu_p)\preccurlyeq(N,\mu_q)$.
\end{proposition}
\begin{proof}
(i) Let $\frak d$ be a definition scheme for $p$. We prove by induction that every
$L_\mu$-formula $\phi(\y)$ is equivalent to a $L(M)$-definable relation.
The main step is to prove the claim for $\int\phi(x,\y)dx$ assuming it holds for $\phi(x,\y)$.
Assume $(M,\mu_p)$ satisfies $\phi(x,\y)=R(x,\y)$ where $R$ is a $L(M)$-definable relation.
Then, for each $\a\in M$ $$\int\phi^M(x,\a)dx=\int R^M(x,\a)dx=p(R(x,\a))=(\frak dR)(\a).$$
So, in $(M,\mu_p)$,\ \ $\int\phi(x,\y)dx$ is equivalent to the relation $\frak dR(\y)$.

(ii) Assume $q$ is non-forking over $p$. Then, $\frak d$ above is a definition scheme for both $p$ and $q$.
So, the procedure of elimination of integral formulas is the same in $M$ and $N$.
In other words, in $(N,\mu_p)$, every $L_\mu$-formula is equivalent to a $L(M)$-definable relation.
In particular, $(M,\mu_p)\preccurlyeq(N,\mu_q)$.
Conversely assume $(M,\mu_p)\preccurlyeq(N,\mu_q)$. Then, for some $\wp$, there is an isomorphism
$f:(M,\mu_p)^{\wp}\rightarrow(N,\mu_q)^{\wp}$ where $f$ fixes $M$ pointwise.
Since $p^\wp$ is non-forking over $p$ and $q^\wp=f(p^\wp)$ is non-forking over $q$, we conclude that $q$
is non-forking over $p$.
\end{proof}

Proposition \ref{def-scheme1} shows that $p(x)$ and its non-forking extensions identify a complete affine theory in $L_\mu$.
Expanding it to the language $L_\mu(M)$, we obtain a theory which eliminates the integral quantifier,
i.e. every formula is equivalent to an integral-free formula.
\bigskip

\noindent{\bf Variations of integration logic}:\\
The essential parts of the integration logic, consisting of the compactness theorem and its consequences,
can be generalized in various ways. We suggest some ideas without giving further details.
The integral operator can be replaced with an arbitrary linear operator.
Let $L$ be a charged language as before. Instead of endowing the metric structure $M$ with a probability charge,
we endow it with a norm $1$ continuous linear operator $\Lambda:\mathbf{C}_b(M)\rightarrow\Rn$ with $\Lambda(1)=1$.
Formulas are defined as before and $$\big(\int\phi(\a,y)dy\big)^M=\Lambda(\phi^M(\a,y)).$$
The argument for ultramean construction is similar and affine compactness theorem is proved by using
the Hahn-Banach extension theorem instead of Corollary \ref{KR}. The logic AL$_{\int}$ explained before is then axiomatized
in the generalized framework by the axiom scheme $$\inf_y\phi(\x,y)\leqslant\int\phi(\x,y)dy.$$
It is also possible to replace unary operators with $n$-ary operators, i.e. operators of the form $\Lambda:{\mathbf C}_b(M^n)\rightarrow\Rn$.
In particular, unary charges may be replaced with $n$-dimensional charges and Proposition \ref{def-scheme1} is generalized to $n$-dimensional types.
To obtain a parallel result for types in first order theories, expansions of first order logic by ultrafilters like in \cite{Veloso} are needed.


\begin{thebibliography}{5}
\bibitem{Aliprantis} C.D. Aliprantis, O. Burkinshaw, \textit{Positive operators}, Springer (2006).
\bibitem{Aliprantis-Inf} C.D. Aliprantis, K.C. Border, \textit{Infinite dimensional analysis}, Springer, 2006.
\bibitem{ACMT} S.M. Bagheri, \textit{Elements of affine model theory},  arXiv:2408.03555v2.
\bibitem{Isomorphism} S.M. Bagheri, \textit{The isomorphism theorem for linear fragments of continuous logic},
Math. Log. Quart. 67, No.2, 193-205 (2021).
\bibitem{Bagheri-Pourmahd2} S.M. Bagheri, M. Pourmahdian, \textit{Continuous integration logic}, arXiv:1910.00191.
\bibitem{Bagheri-Pourmahd1} S.M. Bagheri, M. Pourmahdian, \textit{The logic of integration}, Arch. Math. Logic. 48 (5), 465-492 (2009).
\bibitem{Ibarlucia} I. Ben-Yaacov, T. Ibarluc\'{i}a, T. Tsankov,
\textit{Extremal models and direct integrals in affine logic}, Preprint arXiv:2407.13344 (2024).
\bibitem{Ibarlucia2} I. Ben-Yaacov, T. Ibarluc\'{i}a, \textit{Stability in affine logic}, Preprint arXiv:2503.10506v2 (2026).
\bibitem{Rao} K.P.S. Bhaskara Rao, M. Bhaskara Rao, \textit{Theory of charges}, Academic Press (1983).
\bibitem{Veloso} Carnielli, W. A. and P. A. S. Veloso, \textit{Ultrafilter Logic and Generic
Reasoning}, Gottlob, G., Leitsch, A. and Mundici, D. (eds.) Computational
Logic and Proof Theory (LNCS 1289): 34-53, Springer-Verlag, Berlin (1997).
\bibitem{Hoover} D.N. Hoover, Probability logic, Annals of Mathematical Logic 14 (1978) 287-313.
\bibitem{Keisler} H.J. Keisler, \emph{Probability quantifiers} in J. Barwise, S. Feferman, Model theoretic logic, Springer-Verlag (1985).
\end{thebibliography}
\end{document}